\documentclass[12pt, twoside]{amsart}

\usepackage[margin=1in]{geometry}
\usepackage{mathtools}
\usepackage{booktabs}
\usepackage{float}
\usepackage{longtable}

\usepackage{xcolor}
\usepackage[colorlinks=true]{hyperref}
\definecolor{darkblue}{RGB}{0,50,150}
\hypersetup{linkcolor=darkblue, citecolor=darkblue, urlcolor=darkblue}

\usepackage{url}

\usepackage{amsmath,amsthm,amsfonts,amssymb}

\newtheorem{theorem}{Theorem}[section]

\newtheorem{corollary}[theorem]{Corollary}
\newtheorem{conjecture}[theorem]{Conjecture}
\newtheorem{remark}[theorem]{Remark}

\theoremstyle{definition}
\newtheorem{definition}[theorem]{Definition}



\title{Prime Quadruplets and Jump Conditions on Arithmetic Functions}
\author{Himaghna Roy Choudhury, Shicheng Wei}

\newcommand{\Cpp}{{C\nolinebreak[4]\hspace{-.05em}\raisebox{.4ex}{\tiny\bf ++}}}
\begin{document}

\begin{abstract}
We provide progress on the characterization of composite integers $n$ that satisfy the jump conditions $\varphi(n+12)=\varphi(n)+12$ and $\sigma(n+12)=\sigma(n)+12$ simultaneously. While it is known that prime quadruplets $(p,p+2,p+6,p+8)$ generate solutions $n=p(p+8)$, the complete characterization remains an open conjecture. We prove that this characterization is complete when $n$ and $n+12$ are both squarefree semiprimes, and that no solution $n$ can be a prime power. Furthermore, a complete search up to $10^{12}$ resulted in no counterexamples to the conjecture. If this conjecture is proven true, and there are infinitely many such solutions, then it can be proved that there are infinitely many prime quadruplets.
\end{abstract}

\maketitle

\section{Introduction}\label{sec:intro}

% A classic topic to study in number theory is the behavior of arithmetic functions, such as $\varphi$, Euler's totient function, and $\sigma$, the sum-of-divisors function, over shifted intervals. For instance, characterizing the solutions to $\sigma(n)=\sigma(n+1)$ remains a well-known open conjecture (see Guy \cite[Section B13]{Guy}). Another equation, namely $\sigma(n+2) = \sigma(n) + 2$, has solutions when $n$ and $n+2$ are twin primes, but composite solutions still exist \cite[Section B13]{Guy}. Other equations may inquire about the relationships between the two arithmetic functions, such as $\varphi(m)=\sigma(n)$, which can be shown to have infinitely many solutions if there are infinite twin primes.

A classic topic to study in number theory is the behavior of arithmetic functions, such as $\varphi$, Euler's totient function, and $\sigma$, the sum-of-divisors function, over shifted intervals. For instance, Graham, Holt, and Pomerance~\cite{GrahamHoltPomerance1999} studied the solutions to $\varphi(n)=\varphi(n+k)$, where $k$ is a fixed positive integer. They constructed families of solutions for even $k$ and obtained an asymptotic formula for the number of solutions under a quantitative form of the prime $k$-tuple conjecture. For the equation $\sigma(n)=\sigma(n+k)$, Weingartner~\cite{Weingartner2011} showed that there are infinitely many solutions for every fixed positive even $k$, assuming Schinzel's hypothesis H.

Yamada~\cite{Yamada2017} studied both equations and obtained upper bounds for the number of solutions outside certain specified families, also considering the more general case of linear arguments. Ford~\cite{Ford2022} proved unconditionally that $\varphi(n)=\varphi(n+k)$ has infinitely many solutions for certain shifts $k$, including some positive even $k\leq3570$. He also showed that the equation $\sigma(n)=\sigma(n+k)$ has infinitely many solutions for a positive proportion of positive integers $k$.

Other equations concern the relationship between the two arithmetic functions, such as $\varphi(m)=\sigma(n)$. If $p$ and $p+2$ are twin primes, then $\varphi(p+2)=p+1=\sigma(p)$, so infinitely many twin primes would give infinitely many solutions. However, Ford, Luca, and Pomerance~\cite{FordLucaPomerance2010} proved unconditionally that the two functions have infinitely many common values. Thus, the infinitude of solutions to this equation does not depend on the twin prime conjecture.

In this paper we study a simultaneous jump condition for both
arithmetic functions. More specifically, in the 2004 arXiv preprint
\emph{Prove or Disprove: 100 Conjectures from the OEIS}
\cite{Stephan}, R. Stephan proposed Conjecture (33), which we
state here with the corrected residue:
\begin{equation}\label{eq:stephan33}
\bigl\{\text{composite } n \;\big|\; \varphi(n+12) = \varphi(n) + 12 \;\wedge\; \sigma(n+12) = \sigma(n) + 12\bigr\}
\;\Longrightarrow\; n \equiv 65 \pmod{72}.
\end{equation}

The preprint gives the residue $64\pmod{72}$, which is a
typographical error. The corrected residue $65\pmod{72}$ 
also appears in Stephan's conjecture in the OEIS entry
\texttt{A056777} \cite{OEIS-A056777}. 
For example, when $n=65$, we have $\varphi(65)=48$,
$\varphi(77)=60$, $\sigma(65)=84$, and $\sigma(77)=96$.
Thus,
\[
\varphi(77)-\varphi(65)=60-48=12,
\qquad
\sigma(77)-\sigma(65)=96-84=12,
\]
so $n=65$ satisfies both jump conditions.

In addition, in the same sequence entry, Jud McCranie conjectured that all solutions satisfying the jump conditions can be completely characterized by prime quadruplets, that is, prime tuples of the form $(p,p+2,p+6,p+8)$. We formulate this into a conjecture below.

\begin{conjecture}\label{McCranie}
    Every composite integer $n$ satisfying both $\varphi(n+12) = \varphi(n) + 12$ and $\sigma(n+12) = \sigma(n) + 12$ is of the form $n=p(p+8)$, where $(p,p+2,p+6,p+8)$ is a prime quadruplet with $p\ge5$. 
\end{conjecture}

Notice that if Conjecture~\ref{McCranie} holds, then Stephan's corrected residue claim follows as a corollary, and can be easily shown with the Chinese Remainder Theorem. Additionally, Conjecture~\ref{McCranie} effectively bridges these jump conditions about $\varphi$ and $\sigma$ with the Hardy--Littlewood prime $k$-tuple conjecture for $k=4$ \cite{HardyLittlewood}. If Conjecture~\ref{McCranie} is true and it can be shown that there are infinitely many solutions to these jump conditions, then it would imply infinitely many prime quadruplets.

In this paper, we prove that the prime quadruplet construction gives every solution in the case where $n$ and $n+12$ are squarefree semiprimes, and that no composite prime power satisfies both jump conditions.

We assume throughout that $n \ge 4$ is composite. The functions $\varphi$ and $\sigma$ are defined on $\mathbb{Z}_{\ge 1}$ in the usual way. Recall that they are multiplicative, where $\varphi(p) = p-1$, $\sigma(p) = p+1$ for primes $p$, and $\varphi(pq) = (p-1)(q-1)$, $\sigma(pq) = (p+1)(q+1)$ for distinct primes $p$ and $q$.

\section{Characterization in the Semiprime Case}\label{sec:semiprime-char}

We now show that within the class of squarefree semiprime solutions, the prime quadruplet construction gives a complete characterization.

\begin{definition}
A \emph{prime quadruplet} is a $4$-tuple of primes of the form $(p,\, p+2,\, p+6,\, p+8)$. The smallest primes $p$ of prime quadruplets form the OEIS sequence \texttt{A007530}, beginning with $5, 11, 101, 191,$ $ 821, 1481, 1871, 2081, \ldots$ \cite{OEIS-A007530}. The Hardy--Littlewood prime $k$-tuple conjecture, which remains open, predicts that there are infinitely many \cite{HardyLittlewood}.
\end{definition}

\begin{theorem}\label{thm:semiprime-char}
Suppose $n = pq$ and $n+12 = p'q'$ are both products of two distinct primes (with $p < q$ and $p' < q'$). Then both jump conditions
\[
\varphi(n+12) - \varphi(n) = 12, \qquad \sigma(n+12) - \sigma(n) = 12
\]
hold if and only if $p \ge 5$, $q = p+8$, and $\{p', q'\} = \{p+2,\, p+6\}$. That is, $(p,\, p',\, q',\, q) = (p,\, p+2,\, p+6,\, p+8)$ is a prime quadruplet.\end{theorem}

\begin{proof}
($\Leftarrow$) Using $\varphi(rs) = (r-1)(s-1)$ and $\sigma(rs) = (r+1)(s+1)$ for distinct primes $r$ and $s$, we get
\begin{align*}
\varphi(n) &= (p-1)(p+7) = p^2 + 6p - 7, \\
\varphi(n+12) &= (p+1)(p+5) = p^2 + 6p + 5, \\
\sigma(n) &= (p+1)(p+9) = p^2 + 10p + 9, \\
\sigma(n+12) &= (p+3)(p+7) = p^2 + 10p + 21.
\end{align*}
By subtracting, we get $\varphi(n+12) - \varphi(n) = 12$ and $\sigma(n+12) - \sigma(n) = 12$ from just polynomial identities in terms of $p$.

($\Rightarrow$) For any squarefree semiprime $rs$ with $r$ and $s$ distinct primes,
\[
\varphi(rs) = (r-1)(s-1)=rs - (r+s) + 1, \qquad \sigma(rs) = (r+1)(s+1) = rs + (r+s) + 1.
\]
Applying this to $n = pq$ and $n+12 = p'q'$:
\begin{align*}
\varphi(n)      &= pq - (p+q) + 1                       &&= n - (p+q) + 1, \\
\varphi(n+12)   &= p'q' - (p'+q') + 1                   &&= (n+12) - (p'+q') + 1, \\
\sigma(n)       &= pq + (p+q) + 1                       &&= n + (p+q) + 1, \\
\sigma(n+12)    &= p'q' + (p'+q') + 1                   &&= (n+12) + (p'+q') + 1.
\end{align*}
Subtracting the first from the second and the third from the fourth:
\begin{align*}
\varphi(n+12) - \varphi(n) &= \bigl[(n+12) - (p'+q') + 1\bigr] - \bigl[n - (p+q) + 1\bigr] = 12 + (p+q) - (p'+q'), \\
\sigma(n+12)  - \sigma(n)  &= \bigl[(n+12) + (p'+q') + 1\bigr] - \bigl[n + (p+q) + 1\bigr] = 12 + (p'+q') - (p+q).
\end{align*}
Both expressions equal $12$ if and only if
\begin{equation}\label{eq:sum-eq}
p + q = p' + q'.
\end{equation}
Now, let $S = p+q = p'+q'$. Then, $p, q$ are the roots of $X^2 - SX + n = 0$, and $p', q'$ are the roots of $X^2 - SX + (n+12) = 0$. The discriminants are
\[
u^2 := S^2 - 4n = (q - p)^2 \ge 0, \qquad v^2 := S^2 - 4(n+12) = (q' - p')^2 \ge 0,
\]
so
\begin{equation}\label{eq:disc-diff}
u^2 - v^2 = 48.
\end{equation}

We now seek integer pairs $(u, v)$ with $u > v \ge 0$ satisfying \eqref{eq:disc-diff}. We first write $48 = (u-v)(u+v)$, which leads us to notice that $u-v$ and $u+v$ must have the same parity, since their sum $2u$ is even. 

If both factors are odd, their product is odd, which contradicts the fact that $48$ is even. If both are even, then let $u - v = 2a$ and $u + v = 2b$ with $a < b$ and $ab = 12$. The factorizations $(a, b) \in \{(1, 12),\, (2, 6),\, (3, 4)\}$ yield
\[
(u, v) = (a + b,\, b - a) \in \{(13, 11),\, (8, 4),\, (7, 1)\}.
\]

We can verify them directly: $7^2 - 1^2 = 48$, $8^2 - 4^2 = 48$, $13^2 - 11^2 = 48$. These are the only integer solutions with $u > v \ge 0$.

Now recall $S = p + q$. We consider separate cases for whether $p = 2$:

\emph{Case 1: $p, q$ both odd.} Then $S = p+q$ is even, and $u = q - p$ must satisfy $S - u = 2p$ (even), forcing $u$ even. Of the three solutions, only $(u, v) = (8, 4)$ has an even $u$. Hence
\[
q - p = 8, \quad q' - p' = 4.
\]
\begin{samepage}
Combined with $p' + q' = p + q = 2p + 8$, this gives us a linear system in $p', q'$:
\[
p' + q' = 2p + 8, \qquad q' - p' = 4.
\]
Adding and subtracting yields $2q' = 2p + 12$ and $2p' = 2p + 4$ respectively. Hence
\[
p' = p + 2, \qquad q' = p + 6.
\]
\end{samepage}
Since $p, \, p', \, q', \, q$ are prime by hypothesis, we know that $(p, p+2, p+6, p+8)$ is a prime quadruplet. In addition, $p \ge 5$, because the smallest prime quadruplet starts at $p = 5$ (if $p = 3$, then it requires $3, 5, 9, 11$ all to be prime, but $9 = 3^2$ is not).

\emph{Case 2: $p = 2$.} Then $n = 2q$ is even, so $n + 12 = p'q'$ is also even. Since $p', q'$ are both prime and $2$ is the only even prime, either $p'$ or $q'$ must be $2$. By the convention $p' < q'$, we have $p' = 2$, which leads to
\[
q' = \frac{n+12}2 = \frac{2q+12}2 = q + 6.
\] 
Condition~\eqref{eq:sum-eq} gives $2 + q = 2 + q'$, i.e., $q = q' = q + 6$, which is impossible. (Alternatively, $\varphi(2q) = q - 1$ and $\varphi(2(q+6)) = q + 5$, with difference $6 \neq 12$.) So this case yields no solutions.
 
This exhausts the cases.
\end{proof}

\begin{remark}
Theorem~\ref{thm:semiprime-char} proves that every solution for which both $n$ and $n+12$ are squarefree semiprimes comes from a prime quadruplet. The remaining problem to address is when at least one of $n$ or $n+12$ is not a squarefree semiprime. Below we show that the prime quadruplet construction gives the residue claimed by Stephan.
\end{remark}

\begin{corollary}\label{cor:residue} If $n = p(p+8)$ where $(p, p+2, p+6, p+8)$ is a prime quadruplet with $p \ge 5$, then $n \equiv 65 \pmod{72}$. Consequently, if Conjecture~\ref{McCranie} holds, all composite solutions to the jump conditions must satisfy $n \equiv 65 \pmod{72}$.
\end{corollary}

\begin{proof}
We use the Chinese Remainder Theorem with $72 = 8 \cdot 9$.

\emph{Modulo $8$.} Since $p$ is an odd prime, $p^2 \equiv 1 \pmod 8$. Therefore
\[
n = p(p+8) \equiv p \cdot p = p^2 \equiv 1 \pmod 8.
\]

\emph{Modulo $9$.} Since $p$ is a prime and $p\ge5$, we know $p \not\equiv 0 \pmod 3$. Similarly, $p+2$ being prime with $p \ge 5$ gives us $p \not\equiv 1 \pmod 3$. Therefore
\[
p \equiv 2 \pmod 3, \quad \text{i.e.,} \quad p \equiv 2,\, 5,\, \text{or } 8 \pmod 9.
\]
For each possible residue, we compute $n = p(p+8) \bmod 9$:
\begin{align*}
p \equiv 2: \quad & p(p+8) \equiv 2 \cdot 10 \equiv 2 \cdot 1 \equiv 2 \pmod 9, \\
p \equiv 5: \quad & p(p+8) \equiv 5 \cdot 13 \equiv 5 \cdot 4 \equiv 20 \equiv 2 \pmod 9, \\
p \equiv 8: \quad & p(p+8) \equiv 8 \cdot 16 \equiv 8 \cdot 7 \equiv 56 \equiv 2 \pmod 9.
\end{align*}
Hence, $n \equiv 2 \pmod 9$ in all cases.

\smallskip
\emph{CRT.} We seek $x \in \mathbb{Z}/72\mathbb{Z}$ with $x \equiv 1 \pmod 8$ and $x \equiv 2 \pmod 9$. Using the CRT, we obtain that $x\equiv65 \pmod {72}$.

If Conjecture~\ref{McCranie} holds, every composite solution has
the form above, so the same congruence follows.
\end{proof}

\section {Non-Existence of Single Prime Power Solutions}\label{sec:SinglePrimePower}

We contribute partial progress to Conjecture~\ref{McCranie} by proving that the case where $n$ is a power of a single prime yields no solutions. We do this by introducing the auxiliary function $Q(m)=m^2-\varphi(m)\sigma(m)$ and doing casework on how many distinct prime factors $m$ has. Finally, we create contradictions in all cases.

\begin{theorem}\label{primepower}
    No prime $p$ and integer $k \ge 2$ satisfy both $\varphi(p^k+12) = \varphi(p^k)+12$ and $\sigma(p^k+12)=\sigma(p^k)+12$.
\end{theorem}

\begin{proof}
    We define $A(m)$, $B(m)$, and $T(m)$ such that
    \[A(m)=m-\varphi(m), \quad  B(m)=\sigma(m)-m, \quad T(m)=B(m)-A(m).\]
    Assume for the sake of contradiction that there exists $n=p^k$ that satisfies $\varphi(p^k+12) = \varphi(p^k)+12$ and $\sigma(p^k+12)=\sigma(p^k)+12$. Let $N=n+12$. Then it follows that
    \begin{align*}
        A(n)&=p^{k-1},   \quad
        B(n)=\frac{p^k-1}{p-1}, \quad
        T(n)=\frac{p^{k-1}-1}{p-1}, \\
        A(N)&=p^{k-1}, \quad
        B(N)=\frac{p^k-1}{p-1}, \quad
        T(N)=\frac{p^{k-1}-1}{p-1}.
    \end{align*}
    So we can immediately see
    \[
        A(n)=A(N), \quad B(n)=B(N), \quad T(n) = T(N).
    \]
    Now let
    \[Q(m)=A(m)B(m)-mT(m)=m^2-\varphi(m)\sigma(m).\]
    It is known that for $\displaystyle m=\prod_{i=1}^s r_i^{e_i}$, where $r_1,r_2,...,r_s$ are distinct primes and $e_i\ge1$,
    \[\frac{\varphi(m)\sigma(m)}{m^2}=\prod_{i=1}^s\left(1-\frac{1}{r_i^{e_i+1}}\right).\]
    We can use this to write
    \[Q(m)=m^2\left(1-\prod_{i=1}^s\left(1-\frac{1}{r_i^{e_i+1}}\right)\right).\]
    
    If $m$ is a power of a single prime, specifically $m=r^e$, then
    \[Q(m)=r^{2e}\left(1-\left(1-\frac{1}{r^{e+1}}\right)\right)=\frac{r^{2e}}{r^{e+1}}=r^{e-1}<m,\]
    so $Q(m)<m$ in this case.
    
    If $m$ has at least two distinct prime factors, then let $r_1$ be the smallest prime factor of $m$ such that $m=r_1^{e_1}u$ and $\gcd(r_1,u)=1$. Because $u$ contains at least one other prime $r_2>r_1$, we know $u\ge r_2 > r_1$, hence
    \begin{equation}\label{r_1}
        m=r_1^{e_1}u>r_1^{e_1+1}.
    \end{equation}
    We can separate $r_1$ from the product in the expression for $Q(m)$:
    \begin{align*}
        Q(m)=m^2\left(1-\left(1-\frac{1}{r_1^{e_1+1}}\right)\prod_{i=2}^s\left(1-\frac{1}{r_i^{e_i+1}}\right)\right).
    \end{align*}
    Now we can make the observation that
    \[1-\left(1-\frac{1}{r_1^{e_1+1}}\right)\prod_{i=2}^s\left(1-\frac{1}{r_i^{e_i+1}}\right)>\frac{1}{r_1^{e_1+1}},\]
    which implies
    \[Q(m)>\frac{m^2}{r_1^{e_1+1}}.\]
    Using (\ref{r_1}), we see that
    \[\frac{m^2}{r_1^{e_1+1}}>m,\]
    so $Q(m)>m$ in this case.
    
    Now we look at $n$ and $N$. Since $n=p^k$, we have 
    \[Q(n)=p^{k-1}\cdot\frac{p^k-1}{p-1}-p^k\cdot\frac{p^{k-1}-1}{p-1}=p^{k-1}<n.\]
    Then we also know 
    \begin{align*}
        Q(N)&=A(N)B(N)-NT(N)\\
        &=A(n)B(n)-(n+12)T(n)\\
        &=\left(A(n)B(n)-nT(n)\right)-12T(n)\\
        &=Q(n)-12T(n).
    \end{align*}
    Note that $T(n)=\frac{p^{k-1}-1}{p-1}>0$, so we can now write
    \begin{equation}\label{Q(N)<N}
        Q(N)=Q(n)-12T(n)<Q(n)<n<N.
    \end{equation}
    
    If $N$ has at least two distinct prime factors, then $Q(N)>N$, which contradicts $Q(N)<N$ from (\ref{Q(N)<N}). If $N$ is a power of a single prime, say $q^\ell$, where $q$ is prime and $\ell\ge1$, then the fact that $A(n)=A(N)$ implies
    \[p^{k-1}=q^{\ell-1}.\]

    If $\ell=1$, then the right-hand side is $1$, which contradicts $p^{k-1}>1$. So $\ell\ge2$. By unique factorization, this forces $p=q$ and $k-1=\ell-1$. Thus $k=\ell$ and
    \[N=q^\ell=p^k=n,\]
    contradicting $N=n+12$. This exhausts all cases.
\end{proof}

% \section{Conditions on odd prime factors}

% We have shown that $n$ cannot be of a single prime power and have conjectured that both $n$ and $n+12$ must be semiprimes. We prove here why at least one of $n$ or $n+12$ cannot contain more than $2$ distinct odd prime factors.
% \begin{lemma}
%     For composite integers $n$ and $n+12$, if $\varphi(n+12)=\varphi(n) +12$, then at least one of $n$ or $n+12$ cannot have more than $2$ distinct odd prime factors.
% \end{lemma}

% \begin{proof}
%     Let $k_1$ be the number of distinct odd prime factors that $n$ has, and $k_2$ be the number of distinct odd prime factors that $n+12$ has. A divisibility property of Euler's totient function states that if an integer $m$ has $k$ distinct odd prime factors, then $2^k$ divides $\varphi(m)$. Using this property, we can obtain
%     \begin{align*}
%         \varphi(n) \equiv 0 &\pmod {2^{k_1}} \\
%         \varphi(n+12) \equiv 0 &\pmod {2^{k_2}}.
%     \end{align*}
%     Recall that the jump condition is $\varphi(n+12)-\varphi(n)=12$. Taking both sides in modulo $2^{\min(k_1, k_2)}$, we get
%     \[0 \equiv 12 \pmod{2^{\min(k_1, k_2)}}.\]
%     For this to hold, we must have $2^{\min(k_1, k_2)} \le 4$, which means $\min(k_1, k_2) \le 2$. 
% \end{proof}

\section{Computational Evidence}\label{sec:empirical}

While we do not provide a complete proof for Conjecture~\ref{McCranie} in this paper, we present the solutions found from a complete search over $n \in [2, 10^{12})$ using a \Cpp{} program \cite{IntegerSequenceTesting}. For each composite $n$ in this range, we computed $\varphi(n)$, $\varphi(n+12)$, $\sigma(n)$, and $\sigma(n+12)$, and tested both jump conditions. The search returned exactly 166 solutions. Every solution is congruent to $65$ modulo $72$, which supports Stephan's corrected observation. Every solution is also a squarefree semiprime $n=p(p+8)$, where $(p,p+2,p+6,p+8)$ is a prime quadruplet. The complete list of solutions is given in Appendix~\ref{app:solutions}. The existing OEIS entry \texttt{A056777} \cite{OEIS-A056777} has records up to $1{,}006{,}475{,}609$. We expand the search range by approximately $1{,}000$ times.

\subsection*{Algorithmic Method}

A flat linear sieve allocates arrays of size $N$. However, this is approximately $240\,\mathrm{GB}$ at $N = 10^{10}$. For $N = 10^{12}$, we use a segmented sieve to avoid excessive memory usage. Using the sieve of Eratosthenes, we first precomputed approximately $78{,}000$ small primes up to $\sqrt{N} \approx 10^6$. Then, the range $[2,N+12)$ is processed in smaller blocks of size $B\approx5\times10^5$, allowing the values at $n+12$ to be computed for every candidate $n\in[2,N)$. Adjacent blocks are handled so that the final twelve required values beyond $N-1$ are included in the computation. For each block $[\ell, \ell + B)$, we use three arrays of length $B$ to track the remaining cofactor, the running totient value $\varphi$, and the running sum of divisors $\sigma$. Each small prime $p$ sieves its multiples within the block. It divides out the full $p$-power and multiplies in the corresponding $\varphi$ and $\sigma$ factors. Any residual factor greater than $1$ after all small primes is a large prime, handled directly. A number $n$ is prime if and only if $\varphi(n) = n - 1$. Each worker required approximately $12\,\mathrm{MiB}$ for its three block arrays. With $P$ concurrent workers, the total memory usage is $O(PB+\pi(\sqrt{N}))$; for the $32$ workers used in our computation, the aggregate block-array memory was approximately $384\,\mathrm{MiB}$. The total work is $O\!\left(N\log\log N+\frac{N}{B}\pi(\sqrt{N})\right)$. For the $10^{12}$ search, the algorithm was parallelized across 32 cores via OpenMP and executed on an AWS \texttt{c7i.8xlarge} instance, completing in $5536\,\mathrm{s}$.

\bigskip
\noindent \textbf{Acknowledgments.} We thank Ralf Stephan and Jud McCranie for the original conjectures. We thank Labos Elemer and Jud McCranie for the data in the OEIS entry \texttt{A056777}. We also thank Jeffrey Shallit for helpful feedback.

\newpage
\appendix
\section{Computed Solutions}\label{app:solutions}

\small
\begin{longtable}{rrrr}
\caption{All composite solutions $n$ up to $10^{12}$ satisfying the simultaneous jump conditions.}
\label{tab:solutions}\\
\toprule 
$n$ & $n$ & $n$ & $n$ \\
\midrule
\endfirsthead

\toprule
$n$ & $n$ & $n$ & $n$ \\
\midrule
\endhead
65 & 14,934,062,009 & 155,878,884,209 & 437,046,599,009 \\
209 & 18,350,766,209 & 158,042,027,009 & 440,345,052,209 \\
11,009 & 20,783,547,209 & 158,209,040,009 & 441,779,562,209 \\
38,009 & 24,735,425,609 & 161,712,558,209 & 444,135,609,209 \\
680,609 & 27,458,147,009 & 162,219,645,209 & 462,801,287,009 \\
2,205,209 & 27,837,254,009 & 169,772,841,209 & 464,108,375,009 \\
3,515,609 & 29,297,457,209 & 175,607,093,009 & 478,483,475,609 \\
4,347,209 & 35,206,893,209 & 177,118,931,009 & 497,257,677,209 \\
10,595,009 & 37,972,368,209 & 182,538,290,009 & 511,446,674,009 \\
12,006,209 & 38,312,190,209 & 195,872,630,609 & 539,453,525,609 \\
31,979,009 & 40,600,235,009 & 197,442,479,009 & 542,233,413,209 \\
89,019,209 & 40,733,330,609 & 204,787,926,209 & 546,704,966,009 \\
169,130,009 & 47,247,543,209 & 214,790,537,009 & 590,123,558,009 \\
244,766,009 & 50,780,369,009 & 216,378,477,209 & 597,567,650,609 \\
247,590,209 & 57,621,602,009 & 218,532,875,609 & 633,289,682,009 \\
258,084,209 & 59,392,127,009 & 220,979,907,209 & 645,523,868,009 \\
325,622,009 & 61,308,236,009 & 227,543,310,209 & 662,701,824,209 \\
357,777,209 & 61,501,520,009 & 240,663,830,609 & 675,955,287,209 \\
377,330,609 & 66,494,358,209 & 245,634,228,209 & 678,522,875,609 \\
441,630,209 & 67,815,972,209 & 250,235,055,209 & 688,443,575,609 \\
496,175,609 & 71,120,889,209 & 260,727,678,209 & 695,080,701,209 \\
640,343,009 & 72,261,504,209 & 269,158,628,009 & 700,694,555,609 \\
1,006,475,609 & 76,200,842,009 & 287,773,238,009 & 715,690,620,209 \\
1,214,174,009 & 81,079,715,009 & 288,127,400,609 & 730,896,755,609 \\
1,917,126,209 & 81,387,531,209 & 291,065,645,009 & 732,282,390,209 \\
2,636,309,009 & 86,621,319,209 & 301,582,197,209 & 736,086,782,009 \\
3,061,962,209 & 87,542,015,609 & 312,721,416,209 & 766,088,820,209 \\
3,967,110,209 & 89,685,275,609 & 317,436,462,209 & 767,402,280,209 \\
4,517,856,209 & 90,297,245,009 & 324,951,302,009 & 776,998,175,609 \\
4,829,555,009 & 91,200,980,009 & 327,933,749,009 & 791,877,515,609 \\
5,216,450,609 & 106,370,561,009 & 343,296,387,209 & 823,365,686,009 \\
5,969,880,209 & 111,840,080,609 & 353,816,780,609 & 865,039,505,609 \\
6,351,293,009 & 116,236,674,209 & 357,215,405,609 & 879,947,183,009 \\
6,568,292,009 & 119,989,496,009 & 368,819,363,009 & 896,174,622,209 \\
6,843,425,609 & 121,093,560,209 & 387,188,840,009 & 911,977,250,609 \\
7,888,104,209 & 125,496,605,009 & 392,658,890,609 & 918,808,517,009 \\
9,573,644,009 & 128,812,799,009 & 399,531,447,209 & 920,592,275,609 \\
9,827,748,209 & 130,476,276,209 & 399,834,905,609 & 953,171,453,009 \\
10,224,243,209 & 140,816,315,009 & 401,277,906,209 & 956,630,705,609 \\
12,065,924,009 & 151,083,803,009 & 401,696,102,009 & 967,164,068,009 \\
13,580,406,209 & 151,760,889,209 & 427,931,847,209 &  \\
14,231,297,009 & 153,871,830,209 & 432,299,675,009 &  \\

\bottomrule
\end{longtable}

\end{document}